\documentclass[reqno,  11pt]{amsart}

\usepackage[top=3cm, bottom=2.5cm, left=2.5cm, right=2.5cm]{geometry}
 \usepackage{lineno}
\usepackage{enumerate}
\usepackage{amssymb,amsmath}
\usepackage[utf8]{inputenc}
\usepackage{hyperref}
\usepackage{amsthm}
\usepackage{xcolor}
\newtheorem{theorem}{Theorem}
\newtheorem{corollary}{Corollary}
\newtheorem{lemma}{Lemma}

\theoremstyle{definition}

\theoremstyle{remark}
\newtheorem{remark}{Remark}

\allowdisplaybreaks

\newcommand{\R}{\mathbb{R}}

\newcommand{\E}{\mathbb{E}}

\newcommand{\1}{\mathbf{1}}

\catcode`@=11 \@addtoreset{equation}{section} \catcode`@=12

\title[Convergence in probability and relative entropy]{Rigorous derivation of the mean-field limit for the signal-dependent Keller-Segel system}
\author{Jinhuan Wang$^{\,1}$, Keyu Li$^{\,1}$, and Hui Huang$^{\,2}$}
\thanks{The work of J. Wang is partially supported by the National Natural Science Foundation of China
(12171218) and Liaoning Provincial Natural
Science Foundation Program (2024-MS-002); H. Huang is partially supported by the starting grant from Hunan University}
\thanks{Corresponding author: Hui Huang}
\begin{document}
\maketitle
\begin{center}
{\footnotesize
1-School of Mathematics and Statistics, Liaoning University, Shenyang, 110036, P. R. China \\
Email:  wangjh@lnu.edu.cn; lky\_lnu@163.com\\
\smallskip
2-School of Mathematics, Hunan University,  Changsha, 410082, P. R. China
\\
Email: huihuang1@hnu.edu.cn
}
\end{center}
\maketitle
\date{}
\begin{abstract}
We rigorously derive a two-dimensional Keller–Segel type system with signal-dependent sensitivity from a stochastic interacting particle model. By employing suitably defined stopping times, we prove that the convergence of the interacting particle system towards the corresponding mean-field limit equations in probability under an algebraic scaling regime which improves upon existing results with logarithmic scaling. Building on this, we apply the relative-entropy method to obtain strong $L^1$ propagation of chaos, and establish an algebraic convergence rate.
\end{abstract}


\section{Introduction}
In this work, we provide a rigorous derivation of the two-dimensional Keller–Segel-type system with signal-dependent sensitivity expressed as
\begin{align}
\begin{cases}\label{pde}
\partial_t u = \Delta(e^{-v}u+ u), \qquad &x\in\R^2,~ t>0,\\
-\Delta v + v =\chi u, \qquad &x\in\R^2,~ t>0,\\
u(x,0) = u_0(x), \qquad &x\in\R^2,
\end{cases}
\end{align}
where $u(x,t)$ denotes cell density and $v(x,t)$ represents the concentration of the chemical signal. The coefficient $\chi>0$ quantifies the signal-dependent strength of chemotactic sensitivity. 

The formulation of the model is grounded in a robust biological rationale. This model describes the formation of stripe patterns through a self-trapping mechanism. Referring
to \cite{fu2012stripe}, we know that this process, which has been extensively studied by using synthetic biological experimental methods, involves Escherichia coli secreting the signaling molecule acyl-homoserine lactone (AHL). At low concentrations of AHL, the bacteria exhibit high motility, characterized by random motion driven by typical swimming and tumbling behavior with little external interference. However, as the AHL concentration increases, the behavior of the bacterial population changes significantly, eventually leading to a macroscopically static state.
The partial differential equation (PDE) \eqref{pde} is derived from the general signal-dependent Keller–Segel model proposed 
\begin{align}
\begin{cases}\label{pde1}
\partial_t u = \Delta(\gamma(v)u),\\
\tau\partial_t v-\Delta v + v = u,\\
u(x,0)=u_0(x),\quad v(x,0)=v_0(x),
\end{cases}
\end{align}
where $\tau$ takes the value of $0$ or $1$. Here, $\gamma(v)$ represents the signal-dependent motility, which satisfies $\gamma'(v)\le 0$. 
A substantial body of research has been devoted to the existence of solutions for the model \eqref{pde1}, particularly in cases where the signal-dependent function $\gamma(v)$ is subject to different conditions.
In 2017, under the assumptions of upper and lower bounds for both $\gamma(v)$ and its derivative $\gamma'(v)$, Tao and Winkler \cite{tao2017effects} proved the existence of global classical solutions for the two-dimensional case, while demonstrating that the system admits only weak solutions in higher dimensions. Fujie and Jiang \cite{fujie2020global} established that when $\gamma(v)$ satisfies the following conditions:
\begin{equation*}
0<\gamma(v)\in C^3[0,\infty),\quad \gamma'(v)\leq 0 \text{ on } [0,\infty),\quad \lim_{v\to\infty}\gamma(v)=0,
\end{equation*}
global classical solutions exist for arbitrary initial data. Their subsequent work \cite{fujie2021boundedness} showed that if the function $\gamma(v)$ satisfies the asymptotic condition $\lim_{v\to\infty} e^{\alpha v}\gamma(v) = +\infty$ for any $\alpha > 0$, then there exists a globally bounded classical solution.
For the case $\gamma(v) = e^{-\alpha v}$ ($\alpha > 0$), the authors in \cite{fujie2020global,fujie2021comparison,jin2020critical} derived a critical mass threshold: solutions remain uniformly bounded when the initial cell mass is below this critical value. In contrast, studies in \cite{burger2021delayed, fujie2020global,fujie2021comparison} indicate that solutions  blow up as time approaches infinity when the initial cell mass exceeds the critical mass. Nevertheless, it is impossible to blow-up in finite time, which is a difference from the Keller-Segel model. 

In their seminal 2025 work on the model \eqref{pde} (with $\gamma(v)=e^{-v}+1$), the reference \cite{BOL2026113712}
rigorously established the well-posedness of solutions to the PDE \eqref{pde} and gave error estimates between \eqref{pde} and its regularized version. The discussion is already so comprehensive that we can cite it directly. They also conducted an in-depth study in which the convergences for the stochastic particle systems of \eqref{pde} inspired us to refine their methodologies. 
Our focus lies particularly on the stochastic particle equations and the mean field limit corresponding to  \eqref{pde}. By introducing the concept of stopping time (achieving convergence in probability in Section 2), we obtained results with faster convergence rates than those in \cite{BOL2026113712}. 
In Section 3, we establish the propagation of chaos in strong senses for stochastic differential equations (SDEs), (where we apply the relative entropy method and derive a higher convergence rate), which builds a foundational bridge between the macroscopic and microscopic systems.
In the following, we propose that the SDE corresponding to \eqref{pde} for $N\in\mathbb{N}$ interacting particles $\{X_{N,i}^\varepsilon(\cdot)\}_{1\le i\le N}$ in $\R^2$ reads as
\begin{align}\label{sde}
\begin{cases}
dX_{N,i}^\varepsilon(t) = \Big(2\exp\Big(-\frac{1}{N}\sum_{j=1}^N\Phi^\varepsilon(X_{N,i}^\varepsilon(t) - X_{N,j}^\varepsilon(t))\Big)+ 2\Big)^{1/2}dB_i(t),\\
X_{N,i}^\varepsilon(0) =\zeta_i,\qquad 1\le i\le N,
\end{cases}
\end{align}
where $0<\varepsilon<1$. Here, we consider a filtered probability space defined by $(\Omega,\mathcal{F}, (\mathcal{F})_{t\ge0}, \mathbb{P})$ and introduce $\{B_i(\cdot)\}_{1\le i\le N}$, a collection of independent $\mathcal{F}_t$-Brownian motions. The initial data $\zeta_1,\zeta_2,\cdots,\zeta_N$ are assumed to be random variables independent and identically distributed (i.i.d.) with the common probability density function $u_0$. The potential $\Phi^\varepsilon$ is given by
$$\Phi^\varepsilon:=\Phi*j^\varepsilon,\quad \Phi:=\chi \tilde{\Phi},\quad j^\varepsilon(x):=\frac{1}{\varepsilon^2}j(\frac{x}{\varepsilon}),$$
where $\chi>0$ describes the strength rate of the signals in the model \eqref{pde} and $\tilde{\Phi}$ is the Yukawa potential \cite[Chapter 6, 6.23]{lieb2001analysis}, which is defined for any $\mu > 0$, by 
\begin{align*}
&\tilde{\Phi}(x) = \int_0^\infty(4\pi t)^{-1}\exp\Big\{-\frac{|x|^2}{4t}-\mu^2t\Big\}dt.
\end{align*}
Based on \cite[Appendix B]{li2023optimal}, the following properties are observed
\begin{align}\label{Phi}
&\|\tilde{\Phi}\|_{L^p(\R^2)}<\infty,~\|\nabla\tilde{\Phi}\|_{L^q(\R^2)}<\infty,\qquad p\in[1,\infty),~q\in[1,2).
\end{align}
Through calculation, the following results are obtained that for any constant $C$ independent of $\varepsilon$,
\begin{align}
&\|\Phi^\varepsilon\|_{W^{1,1}(\R^2)}\le\|\Phi\|_{W^{1,1}(\R^2)}\|j^\varepsilon\|_{L^1(\R^2)}\le C,\label{D1phi}\\
&\|\Phi^\varepsilon\|_{W^{1,\infty}(\R^2)}\le\|\Phi\|_{W^{1,1}(\R^2)}\|j^\varepsilon\|_{L^\infty(\R^2)}\le \frac{C}{\varepsilon^2},\label{phi}\\
&\|D^2\Phi^\varepsilon\|_{L^\infty(\R^2)}\le \|\nabla\Phi\|_{L^1(\R^2)}\|\nabla j^\varepsilon\|_{L^\infty(\R^2)}\le \frac{C}{\varepsilon^3}\label{D2phi}.
\end{align}
The rigorous derivation of PDEs from stochastic particle systems represents a fundamental challenge in mathematical physics and applied analysis. 
A particularly influential framework for such derivations is the theory of moderately interacting particles, introduced in the pioneering works of Oelschl\"ager  \cite{oelschlager1985law, oelschlager1987fluctuation, oelschlager1989derivation},
where the Law of Large Numbers for interacting diffusions was established, together with central-limit-type fluctuation results.
This approach, characterized by a specific scaling of the interaction potential that preserves nonlocality while weakening pointwise strength, has since been extended to various biological and physical models. For example, Stevens \cite{stevens2000derivation} derived the Keller–Segel chemotaxis model from a moderately interacting system, overcoming the lack of ellipticity, a common assumption in earlier mean field limits, by introducing novel analytical techniques. 

In recent years, significant progress has been made in handling singular interactions, such as Coulomb and Riesz potentials, which arise naturally in models of collective behavior. Lazarovici and Pickl \cite{lazarovici2017mean} analyzed a particle system incorporating a regularized potential and random initial conditions. Serfaty and Duerinckx \cite{serfaty2020mean} developed the modulated energy method to treat Coulomb-type flows, providing a powerful tool for the analysis of the mean field without confinement. 
Subsequent studies on related models employing diverse regularization strategies have further expanded this line of inquiry, as documented in \cite{boers2016mean,bolley2011stochastic,chen2017mean,chen2020combined}.
Currently, the relative entropy method, advanced by Jabin and Wang \cite{jabin2017mean} and further applied by Bresch et al.\cite{bresch2023mean} and Chen et al.\cite{chen2025quantitative,chen2025mean}, has enabled quantitative estimates of convergence for singular attractive kernels. These techniques have been instrumental in the establishment of various modes of convergence, including convergence in the Wasserstein distance (Carrillo et al.\cite{carrillo2019propagation}) and convergence in probability Pickl et al. \cite{huang2017error,huang2020mean,lazarovici2017mean}, strong $L^1$ convergence (Chen et al. \cite{chen2025quantitative} and \cite{olivera2020quantitative,olivera2023quantitative}). More recently, the moderate interaction framework has been instrumental in the derivation of cross-diffusion models; see, for example, \cite{carrillo2024interacting,chen2021rigorous, chen2019rigorous,li2024convergence,liu2016propagation, liu2019propagation}. 
It is also noted that the relative entropy is very useful when analyzing other systems, such as random-batch approximated interacting particle systems and various numerical schemes for SDEs. Huang et al. \cite{huang2025mean} prove a uniform-in-time relative entropy error bound for the Random Batch Method and improve its time-step convergence rate. Li et al. \cite{li2025sharp} establish uniform-in-time error bounds and invariant measure approximation rates for Stochastic Gradient Langevin Dynamics sampling under mild assumptions. In addition, the authors in \cite{li2025relative} investigate the implicit Langevin Monte Carlo method for one-sided Lipschitz drifts, and establish its discretization error bounds as well as geometric ergodicity via PDE techniques and coupling arguments.

To study the relationship between the macroscopic system \eqref{pde} and its microscopic system \eqref{sde}, 
we need to introduce an intermediate particle system (Mean-Field equation). A key object in the mean-field limit is the empirical measure, a random probability measure defined as
$$\mu^\varepsilon_N(t)=\frac{1}{N}\sum_{i=1}^N\delta_{X_{N,i}^\varepsilon(t)},\quad t>0,$$
where $\delta$ is the Dirac delta distribution. The work \cite{olivera2023quantitative} shows that $\mu^\varepsilon_N(t)$ converges to the following PDE solution $u^\varepsilon$. For fixed $\varepsilon>0$, the particle system \eqref{sde} propagates chaos in the many-particle limit $N\to\infty$ towards the non-linear SDE system:
\begin{align}\label{mfs}
\begin{cases}
d\overline{X}_i^\varepsilon(t) = \big(2\exp\big(-\Phi^\varepsilon * u^\varepsilon( \overline{X}_i^\varepsilon,t)\big)+2\big)^{1/2} dB_i(t),\\
\overline{X}_i^\varepsilon(0) = \zeta_i,\qquad 1\le i\le N,
\end{cases}
\end{align}
where
$u^\varepsilon(\cdot,t)$ is the density function of i.i.d. random processes
$\overline{X}_1^\varepsilon(t),\cdots, \overline{X}_N^\varepsilon(t)$. 
And the initial data $\{\zeta_i\}_{i=1}^N$
is i.i.d. with the common probability density function $u^\varepsilon(x,0)$. Using It\^{o}'s formula, the density function $u^\varepsilon$ satisfies the so-called intermediate nonlocal problem, namely,
\begin{align}
\begin{cases}\label{rpde}
\partial_t u^\varepsilon = \Delta(e^{-v^\varepsilon}u^\varepsilon+ u^\varepsilon),\qquad &x\in\R^2,~ t>0,\\
-\Delta v^\varepsilon + v^\varepsilon =\chi u^\varepsilon * j^\varepsilon,\qquad &x\in\R^2,~ t>0,\\
u^\varepsilon(x,0) = u_0*j^\varepsilon(x), \qquad &x\in\R^2.
\end{cases}
\end{align}

In this paper, our argument relies on the Lemma \ref{weak solution} from \cite{BOL2026113712}, which establishes the existence and uniqueness of the weak solution $u^\varepsilon$ to the above system \eqref{rpde} and provides the $L^\infty$ estimate for $\nabla\log u^\varepsilon$.
\begin{lemma}\label{weak solution}
\textup{(Existence of weak solution, \cite[Theorem 2, Lemma 16, and Lemma 18]{BOL2026113712})}.
Let $u_0\ge0$ be an initial probability density that satisfies
\begin{eqnarray*}
&u_0\log u_0\in L^1(\R^2), \quad \nabla\log u_0\in W^{1,q}(\R^2)~(q>2),\\
&u_0\in L^1(\R^2,|x|^2dx)\cap L^p(\R^2)~(1\le p\le \infty).   
\end{eqnarray*}
Furthermore, assume $\|\nabla\log u_0\|_{W^{1,q}(\R^2)}=:M_0$ and $\chi<4/c_*$, where $c_*$ is the optimal constant in the Gagliardo-Nirenberg inequality: $\|\omega\|_{L^4(\R^2)}^4\le c_*\|\omega\|_{L^2(\R^2)}^2\|\nabla\omega\|_{L^2(\R^2)}^2$. Then, for any $T > 0$, the problems \eqref{pde} and \eqref{rpde} admit nonnegative weak solutions $(u, v)$ and $(u^\varepsilon, v^\varepsilon)$ in $\mathbb{R}^2 \times (0, T)$ that satisfied
\begin{eqnarray}\label{uniformweaksolution}
&\|u^\varepsilon\|_{{L^2(0,T;H^1(\R^2))}\cap L^\infty(0,T;L^p(\R^2))}\le C,\quad\||x|^2u^\varepsilon\|_{L^\infty(0, T; L^1(\R^2))}\le C,
\end{eqnarray}
where $C$ is a constant independent of $\varepsilon$. Moreover, there exists $T^* \in (0, T)$ sufficiently small such that the following estimate holds
\begin{align}\label{uniformweaksolution1}
\|\nabla\log u^\varepsilon\|_{L^\infty(0, T^*; W^{1,q}(\R^2))}\le C.    
\end{align}
\end{lemma}

\begin{remark}
   We note that an explicit lower bound for $T^*$ can be derived directly by satisfying both the invariance and contraction conditions of the fixed-point argument in \cite[Lemma 18]{BOL2026113712} simultaneously:
\begin{equation*}
    T^* = \min \left( \frac{\ln 2}{C_1(M_0^2+1)}, \frac{M_0^p(2^{p-1} - 1)}{C_2}, \frac{\ln 2}{C_3(M_0^2+1)}, \frac{1}{2 C_4 M_0^p} \right)\,,
\end{equation*}
where $C_1, C_2, C_3, C_4 > 0$ are fixed constants derived from the Sobolev embeddings and H\"older inequalities in the preceding steps.
Furthermore, this estimate remains local-in-time because the bound grows exponentially with the initial data $M_0$. Lacking a global a priori dissipative estimate for $\nabla \log u^\epsilon$, a continuation argument cannot be applied, as the sequential local existence intervals would systematically shrink.
\end{remark}


Using the above lemma, we obtain our first main result: convergence in probability of the mean-field limit under algebraic scaling.

\begin{theorem}\label{propagation of chaos}
The assumptions of Lemma \ref{weak solution} still hold. Let $\{X_{N,i}^\varepsilon\}_{1\le i\le N}$ and $\{\overline{X}_i^\varepsilon\}_{1\le i\le N}$ be the solutions to systems \eqref{sde} and \eqref{mfs}, respectively. Then 
for any parameters $0<\theta<\frac{1}{2}$ and $0<\alpha<\frac{\theta}{2}$, and integers $k, m\in \mathbb{N}^+$ satisfying $k>\frac{1}{\theta-\alpha}$ and $m>\frac{\alpha k+2}{1-2\theta}$, there exist a constant $C(m,k,T)>0$ and parameters $\gamma,\eta>0$ satisfying the bounds
\begin{align}
&0<\gamma<\min\Big\{\frac{\alpha}{3},\frac{-\alpha k +m(1-2\theta)-2}{2k+4m}\Big\},\label{gamma} \\
&0<\eta\le\min\big\{k(\theta-\alpha)-1, -\gamma(2k+4m)-\alpha k +m(1-2\theta)-2\big\}\label{eta}
\end{align}
such that for all $0\le t\le T$,
\begin{align}\label{Pe}
\mathbb{P}\big(\max_{i=1,\cdots,N}\big|(X_{N,i}^\varepsilon - \overline{X}_i^\varepsilon)(t)\big|> N^{-\alpha}\big)\le C(m,k,T) N^{-\eta}, 
\end{align}   
where the cut-off parameter satisfies
$\varepsilon\sim N^{-\gamma} $.
\end{theorem}

\begin{remark}
To see the feasible range of values for $\gamma$ and $\eta$ more clearly, we provide a special case. Let us choose $\theta=0.35$ and $\alpha=0.1$. 

First, we select an integer $k > \frac{1}{0.35 - 0.1} = 4$. We choose $k=5$. 

Next, we select an integer $m > \frac{0.1(5) + 2}{1 - 2(0.35)} = \frac{2.5}{0.3} \approx 8.33$. We choose $m=10$. 

Substituting these choices into the upper bound for $\gamma$ yields:
$$0<\gamma<\min\Big\{\frac{0.1}{3}, \frac{-0.5 + 10(0.3) - 2}{10 + 40}\Big\} = \min\Big\{0.033, \frac{0.5}{50}\Big\} = 0.01.$$
We may thus assume $\gamma=0.005$. Finally, we evaluate the upper bound for $\eta$:
$$0<\eta\le\min\big\{5(0.25)-1, -0.005(50) - 0.5 + 3 - 2\big\} = \min\{0.25, 0.25\} = 0.25.$$
Therefore, for this configuration, we can safely choose a convergence rate of $\eta=0.2$ while maintaining an algebraic cut-off scaling of $\varepsilon \sim N^{-0.005}$.
\end{remark}

The  proof of Theorem \ref{propagation of chaos}, given in Section 2, follows the approach of \cite{lazarovici2017mean}, adopting a superior algebraic scaling with respect to the maximum norm of the trajectory, which \cite{BOL2026113712} only achieved results under logarithmic scaling. 
Since the diffusion coefficients in the stochastic models \eqref{sde} and \eqref{mfs} given the nonlinear dependence of the diffusion coefficients on inter-particle interactions, the use of the Burkholder-Davis-Gundy inequality becomes indispensable. The key ingredient of our proof is the introduction of stopping times, a technique that facilitates the application of the Law of Large Numbers. 

Based on the uniform $L^\infty(0,T^*;H^2(\R^2))$ bound for $\nabla\log u_\varepsilon$ in Lemma \ref{weak solution} and the propagation of chaos established in Theorem \ref{propagation of chaos}, our next main result addresses the strong $L^1$ convergence for the propagation of chaos. The proof is mainly presented in Section 3, employing the relative entropy method \cite{jabin2018quantitative} as recently refined in \cite{chen2025quantitative}.

\begin{theorem}\label{Propagation of chaos in the strong sense}\textup{(Propagation of chaos in the strong sense).} 
Under the assumptions of Theorem \ref{propagation of chaos}, let $r\in \mathbb{N^+}$, $u_{N,r}^\varepsilon(t, x_1,\cdots, x_r)$ be the $r$-th marginal density of the joint density $u_N^\varepsilon(t, x_1,\cdots, x_N)$ of $\{X_{N,i}^\varepsilon\}_{1\le i\le N}$, and $u^{\varepsilon \otimes r}(t, x_1,\cdots, x_r)$ be the tensor product of the solutions $u^{\varepsilon}$ to the model \eqref{rpde}. Then for parameters $0<\theta<1/2$, $0 <\alpha<\frac{\theta}{2}$, and $k,m\in \mathbb{N}^+$ satisfying $k > \frac{1}{\theta-\alpha}$ and $m \ge \frac{\alpha k +2}{1-2\theta}$, there exist a time $T^*\in(0,T)$, a constant $C(m,r,T)>0$, and a parameter $\beta>0$ satisfying the bound $1<\beta\le\min\{\frac{2\alpha}{\gamma}-6,\frac{\eta}{\gamma}-4\}$ such that 
\begin{align}\label{Ps111}
\|u_{N,r}^{\varepsilon}(t) - u^{\varepsilon\otimes r}(t)\|_{L^\infty(0,T^*;L^1(\R^{2r}))}^2 \le C(r,m,T)\varepsilon^{\beta}, 
  \end{align}
where the cut-off parameter satisfies
$\varepsilon\sim N^{-\gamma} $,
and the parameter $\gamma$ defined by
$$0<\gamma<\min\{\frac{2\alpha}{7},\frac{-\alpha k + m(1-2\theta)-2}{2k+4m}\},$$
and the parameter $\eta$ satisfies
$$5\gamma<\eta<\min\big\{k(\theta-\alpha)-1, -\gamma(2k+4m)-\alpha k +m(1-2\theta)-2\big\}.$$ 
\end{theorem}
\begin{remark}\label{remark2}
To illustrate the feasible range of parameters, we provide an explicit example. 
Set $\theta = 0.4$ and $\alpha = 0.1$. First, we must choose an integer $k > \frac{1}{\theta-\alpha} = \frac{1}{0.3} \approx 3.33$, so we take $k=4$.
Next, we choose an integer $m > \frac{\alpha k+2}{1-2\theta} = \frac{0.4+2}{0.2} = 12$, so we take $m=13$. 
Then the bound for $\gamma$ is
\begin{align*}
0<\gamma<&\min\left\{\frac{2\alpha}{7},\frac{-\alpha k + m(1-2\theta)-2}{2k+4m}\right\}\\
=&\min\left\{\frac{0.2}{7}, \frac{-0.4 + 13(0.2)-2}{8+52}\right\} \approx \min\{0.0285, 0.00333\}=0.00333.
\end{align*}
Taking $\gamma=0.001$ yields
\[
0.005=5\gamma<\eta<\min\big\{k(\theta-\alpha)-1, -\gamma(2k+4m)-\alpha k +m(1-2\theta)-2\big\}
=\min\{0.2, 0.14\}=0.14.
\]
We thus choose $\eta=0.05$. Finally, we evaluate the convergence rate $\beta$:
\[
1<\beta\le\min\left\{\frac{2\alpha}{\gamma}-6,\frac{\eta}{\gamma}-4\right\}
=\min\{194, 46\}=46.
\]
Thus, we achieve an explicit algebraic convergence rate of $\mathcal{O}(\varepsilon^{46})$.
\end{remark}
The prior work by Chen in \cite[Theorem 5]{BOL2026113712} also demonstrated the propagation of chaos in the strong $L^1$ sense, resulting in the subsequent finding 
\begin{align*}
\|u_{N,r}^{\varepsilon}(t) - u^{\varepsilon\otimes r}(t)\|_{L^\infty(0,T^*;L^1(\R^{2r}))}^2  \le C(r,m,T)\varepsilon,\quad \varepsilon\sim (\lambda\log N)^{-\frac{1}{4}}. 
\end{align*}
However, our approach takes advantage of the Law of Large Numbers embodied in Lemma \ref{law of large number} and the convergence of probability in Theorem \ref{propagation of chaos}, which subsequently leads to a more refined estimate compared to \cite{BOL2026113712}. For the convenience of comparison, we substitute the special case of the parameters obtained in Remark \ref{remark2} into \eqref{Ps111}
\begin{align*}
\|u_{N,r}^{\varepsilon}(t) - u^{\varepsilon\otimes r}(t)\|_{L^\infty(0,T^*;L^1(\R^{2r}))}^2 \le C(r,m,T)\varepsilon^{46},\quad \varepsilon\sim N^{-0.001}. 
  \end{align*}
Obviously, we have obtained a  faster convergence rate.



The article is organized as follows. In section 2, we establish the propagation of chaos, which corresponds to the convergence in probability of solutions to the stochastic differential equations \eqref{sde} and \eqref{mfs}. Section 3 further derives quantitative propagation of chaos result in the strong sense by applying the relative entropy method.

\section{Convergence in probability}

In this section, we prove Theorem \ref{propagation of chaos}, which establishes the convergence in probability for the interacting particle system of \eqref{pde} and its mean-field dynamics, based on the estimates of weak solutions in Lemma \ref{weak solution}. In addition, to prove this theorem, we provide the following results.

First, based on the existence of continuous and square-integrable stochastic processes given by \cite[Theorem 2.9, page 289]{sznitman2006topics} and the strong uniqueness from \cite[Theorem 2.5, page 287]{sznitman2006topics}, we demonstrate that the SDEs \eqref{sde} and \eqref{mfs} admit unique strong solutions $\{X_{N,i}^\varepsilon(t)\}_{1\le i\le N}$ and $\{\overline{X}_i^\varepsilon(t)\}_{1\le i\le N}$, whose distributions have a density function $u^\varepsilon(t)$ that constitutes a weak solution of the PDE \eqref{rpde}. Furthermore, applying \cite[Theorem 2.3.1]{nualart2006malliavin}, we conclude that the law of $\overline{X}_i^\varepsilon(t)$ is absolutely continuous with respect to the Lebesgue measure, thus establishing the existence of the density function $u^\varepsilon(t)$. 

We then invoke the following important lemma from \cite{chen2024fluctuations}, which describes the Law of Large Numbers. In \cite{chen2024fluctuations}, the authors provide two estimations in expectation and probability for i.i.d. processes $\overline{X}_i^\varepsilon(t)$ and their associated density $u^\varepsilon (t)$. We present this result in the following lemma.
\begin{lemma}\label{law of large number}
\textup{(Law of Large Numbers, \cite[Lemma 7]{chen2024fluctuations}).}
Let $\{\overline{X}_i^\varepsilon\}_{1\le i\le N}$ be i.i.d. random variables with the common density function $u^\varepsilon$. Given $0<\theta<\frac{1}{2}$ and $\psi^\varepsilon\in L^\infty(\R^2;\R)$, we define the random variables 
\begin{align*}
h_{ij}:= \psi^\varepsilon\big(\overline{X}_i^\varepsilon -\overline{X}_j^\varepsilon\big)-\psi^\varepsilon*u^\varepsilon\big(\overline{X}_i^\varepsilon\big),~~~~i,j=1,\cdots,N,
\end{align*}
and the set 
\begin{align}\label{B}
\mathcal{B}_{\theta, \psi^\varepsilon}^N:= \bigcup_{i=1}^N\{\omega\in\Omega:|\frac{1}{N}\sum_{j=1}^N h_{ij}| > N^{-\theta}\}. \end{align}
Then, for every $m\in\mathbb{N}$, there exists $C(m)>0$ such that
\begin{align}
&\E(|\frac{1}{N}\sum_{j=1}^Nh_i|^{2m})\le C(m)\|\psi^\varepsilon\|_{L^\infty(\R^2)}^{2m}N^{-m},\label{Elnl}\\
&\mathbb{P}(\mathcal{B}_{\theta,\psi^\varepsilon}^N)\le C(m)\|\psi^\varepsilon\|_{L^\infty(\R^2)}^{2m}N^{2m(\theta-1/2)+1}.\label{plnl}
\end{align}
\end{lemma}

Now, we give the proof of Theorem \ref{propagation of chaos} based on the Law of Large Numbers in Lemma \ref{law of large number}. To begin, we define a stopping time to facilitate the application of the Law of Large Numbers. 
Then, it inquires about the application of the Burkholder-Davis-Gundy inequality, since the diffusion coefficients in the stochastic models \eqref{sde} and \eqref{mfs} depend nonlinearly on the interactions between the individuals. Finally, following a similar strategy given in \cite{chen2024fluctuations},  we estimate the probability error of the difference $X_{N,i}^\varepsilon - \overline{X}_i^\varepsilon$. The Markov inequality proves to be essential throughout this process.

\begin{proof}[Proof of Theorem \ref{propagation of chaos}]
We begin by defining the key concepts needed in our proof. Let $0<\theta<\frac{1}{2}$ and take $0<\alpha<\frac{\theta}{2}$. First, we define the stopping time
\begin{align}\label{stopingtime}
&\tau_\alpha(\omega):=\inf\big\{t\in(0,T):\max_{i=1,\cdots,N}|(X_{N,i}^\varepsilon - \overline{X}_i^\varepsilon)(\omega,t)|\ge N^{-\alpha}\big\}.
\end{align}
By using the Markov inequality, it follows that for any $k>0$
\begin{align}\label{MK}
\mathbb{P}\big(\max_{i=1,\cdots,N}\big|(X_{N,i}^\varepsilon - \overline{X}_i^\varepsilon)(t)\big|> N^{-\alpha}\big)
\le\ & \mathbb{P}\big(\max_{i=1,\cdots,N}\big|(X_{N,i}^\varepsilon - \overline{X}_i^\varepsilon)(t\land\tau_\alpha)\big|\ge N^{-\alpha}\big)\nonumber\\
\le\ & \sum_{i=1}^N\mathbb{P}\big(\big|(X_{N,i}^\varepsilon - \overline{X}_i^\varepsilon)(t\land\tau_\alpha)\big|\ge N^{-\alpha}\big)\nonumber\\
\le\ &
\sum_{i=1}^N N^{\alpha k}\mathbb{E}\big[\big|(X_{N,i}^\varepsilon - \overline{X}_i^\varepsilon)(t\land\tau_\alpha)\big|^k\big]\nonumber\\
\le\ &
N^{1+\alpha k}\mathbb{E}\big[\frac{1}{N}\sum_{i=1}^N\big|(X_{N,i}^\varepsilon - \overline{X}_i^\varepsilon)(t\land\tau_\alpha)\big|^k\big]\notag \\
=:&N^{1+\alpha k} \mathbb{E}\big(Q_\alpha^k(t)\big).
\end{align}

In the following, we focus on analyzing the expectation of $Q_\alpha^k(t)$. 
Applying the Burkholder-Davis-Gundy inequality \cite[Theorem 3.28]{karatzas1998brownian} and the mean value theorem, we have the following for $i=1,\cdots, N$,
\begin{align}
\E\big(Q_\alpha^k(t)\big)
=\ &  \E\Big(\frac{1}{N}\sum_{i=1}^N\Big|\int_0^{t\land\tau_\alpha}\Big(2\exp\big(-\frac{1}{N}\sum_{j=1}^N\Phi^\varepsilon(X_{N,i}^\varepsilon(s) -X_{N,j}^\varepsilon(s))\big)+2\Big)^\frac{1}{2}\nonumber\\
&-\Big(2\exp\big(-\Phi^\varepsilon*u^\varepsilon(\overline{X}_i^\varepsilon(s),s)\big) +2\Big)^\frac{1}{2}dB_i(s)\Big|^k \Big)\nonumber\\
\le\ & C(T) \E\Big(\int_0^{t\land\tau_\alpha}\frac{1}{N}\sum_{i=1}^N\Big|\Big(2\exp\big(-\frac{1}{N}\sum_{j=1}^N\Phi^\varepsilon(X_{N,i}^\varepsilon(s) -X_{N,j}^\varepsilon(s))\big)+2\Big)^\frac{1}{2}\nonumber\\
&-\Big(2\exp\big(-\Phi^\varepsilon*u^\varepsilon(\overline{X}_i^\varepsilon(s),s)\big) +2\Big)^\frac{1}{2}\Big|^k ds\Big)\nonumber\\
\le\ & C(T)\E\Big(\int_0^{t\land\tau_\alpha}\frac{1}{N}\sum_{i=1}^N\Big|\frac{1}{N}\sum_{j=1}^N\Phi^\varepsilon\big(X_{N,i}^\varepsilon(s) -X_{N,j}^\varepsilon(s)\big)-\Phi^\varepsilon*u^\varepsilon\big(\overline{X}_i^\varepsilon(s),s\big)\Big|^k ds\Big)\nonumber\\
\le \ & C(T)\E\Big(\int_0^{t\land\tau_\alpha}\frac{1}{N}\sum_{i=1}^N \big(|I_{1,i}(s)|^k + |I_{2,i}(s)|^k\big)ds \Big).\label{ES}
\end{align}
Here, the two components are given by
\begin{align}\label{first term}
&I_{1,i}(s)=\frac{1}{N}\sum_{j=1}^N\Big(\Phi^\varepsilon(X_{N,i}^\varepsilon(s)-X_{N,j}^\varepsilon(s)) - \Phi^\varepsilon(\overline{X}_i^\varepsilon(s) -\overline{X}_j^\varepsilon(s)) \Big)
\end{align}
and
\begin{align*}
&I_{2,i}(s)= \frac{1}{N}\sum_{j=1}^N\Phi^\varepsilon(\overline{X}_i^\varepsilon(s) -\overline{X}_j^\varepsilon(s))  -\Phi^\varepsilon*u^\varepsilon(\overline{X}_i^\varepsilon(s), s).
\end{align*}

The term $I_{2,i}(s)$ is estimate by Law of Large Numbers argument. We apply Lemma \ref{law of large number} with $\psi^\varepsilon = \Phi^\varepsilon$ and $0<\theta<\frac{1}{2}$, take $I_{2,i}=\frac{1}{N}\sum_{j=1}^N h_{ij}$. With the notation of Lemma \ref{law of large number} we have \begin{align}\label{BB}
\mathcal{B}_{\theta, \Phi^\varepsilon}^N(s):= \bigcup_{i=1}^N\{\omega\in\Omega:|I_{2,i}(s)| > N^{-\theta}\}.
\end{align}
Regarding the second term $I_{2,i}(s)$, we split it as follows
\begin{align}\label{EI2i}
&\E\Big(\int_0^{t\land\tau_\alpha}\frac{1}{N}\sum_{i=1}^N|I_{2,i}(s)|^k ds\Big)\nonumber\\
\le\ & \E\Big(\int_0^{t\land\tau_\alpha}\frac{1}{N}\sum_{i=1}^N|I_{2,i}(s)|^k \1_{(\mathcal{B}_{\theta,\Phi^\varepsilon}^N)^c}ds\Big) 
+ \E\Big(\int_0^{t\land\tau_\alpha}\frac{1}{N}\sum_{i=1}^N|I_{2,i}(s)|^k \1_{\mathcal{B}_{\theta,\Phi^\varepsilon}^N}ds\Big)\nonumber\\
=:\ & I_{21}(t) + I_{22}(t).
\end{align}

For $I_{21}(t)$, we apply Lemma by the definition of $\mathcal{B}_{\theta,\Phi^\varepsilon}^N$ in \eqref{BB}, we know that 
\begin{align*}
|I_{2,i}|\le N^{-\theta}~~ \text{in}~ (\mathcal{B}_{\theta,\Phi^\varepsilon}^N)^c.
\end{align*}
Hence, we obtain for any $0<t<T$
\begin{align}\label{I21}
I_{21}(t)\le \int_0^{t\land\tau_\alpha} N^{-k\theta } ds\le TN^{-k\theta }.
\end{align}

For $I_{22}(t)$, by using the Law of Large Numbers \eqref{plnl} and the properties of $\Phi^\varepsilon$ in \eqref{phi} and H\"{o}lder's inequality, we have
\begin{align}\label{I22}
I_{22}(t)\le\ & C\|\Phi^\varepsilon\|_{L^\infty(\R^2)}^k\E\Big(\int_0^{t\land\tau_\alpha}\1_{\mathcal{B}_{\theta,\Phi^\varepsilon}^N}ds\Big)\nonumber\\
\le\ & C(T) \|\Phi^\varepsilon\|_{L^\infty(\R^2)}^k\sup_{0<s<T}\mathbb{P}(\mathcal{B}_{\theta,\Phi^\varepsilon}^N)
\le\  C(m,T)\|\Phi^\varepsilon\|_{L^\infty(\R^2)}^{k+2m} N^{2m(\theta-1/2)+1}\nonumber\\
\le\ & C(m,T)\varepsilon^{-2k-4m}N^{2m(\theta-1/2)+1}. 
\end{align}
Substituting \eqref{I21} and \eqref{I22} into \eqref{EI2i}, we have
\begin{align}\label{I2i}
&\E\Big(\int_0^{t\land\tau_\alpha}\frac{1}{N}\sum_{i=1}^N|I_{2,i}(s)|^k ds\Big)
\le\  C(m,T)(N^{-k\theta} + \varepsilon^{-2k-4m}N^{2m(\theta-1/2)+1}).
\end{align}

For the first part in \eqref{ES}, through the Taylor expansion, we derive 
\begin{align}\label{EI1i}
&\E\Big(\int_0^{t\land\tau_\alpha}\frac{1}{N}\sum_{i=1}^N|I_{1,i}(s)|^k ds\Big)\nonumber\\
=\ & \E\Big(\int_0^{t\land\tau_\alpha}\frac{1}{N}\sum_{i=1}^N\Big|\frac{1}{N}\sum_{j=1}^N\big(\Phi^\varepsilon(X_{N,i}^\varepsilon-X_{N,j}^\varepsilon)(s) - \Phi^\varepsilon(\overline{X}_i^\varepsilon -\overline{X}_j^\varepsilon)(s) \big)\Big|^kds\Big)\nonumber\\
\le\ & C \E\Big(\int_0^{t\land\tau_\alpha}\frac{1}{N}\sum_{i=1}^N\Big|\frac{1}{N}\sum_{j=1}^N\nabla\Phi^\varepsilon(\overline{X}_i^\varepsilon -\overline{X}_j^\varepsilon)(s) \times\big((X_{N,i}^\varepsilon- \overline{X}_i^\varepsilon)(s) - (X_{N,j}^\varepsilon-\overline{X}_j^\varepsilon)(s) \big)\Big|^kds\Big)\nonumber\\
&+ C\|D^2\Phi^\varepsilon\|_{L^\infty(\R^2)}^k \E\Big(\int_0^{t\land\tau_\alpha}\frac{1}{N}\sum_{i=1}^N\frac{1}{N}\sum_{j=1}^N\Big|(X_{N,i}^\varepsilon- \overline{X}_i^\varepsilon)(s) - (X_{N,j}^\varepsilon-\overline{X}_j^\varepsilon)(s) \Big|^{2k}ds\Big)\nonumber\\
=:\ & C(I_{11}(t) + I_{12}(t)).
\end{align}
By the property of $D^2\Phi^\varepsilon$ in \eqref{D2phi}, $I_{12}(t)$ can be estimated as
\begin{align}\label{I12}
I_{12}(t)
\le\ & 
C\varepsilon^{-3k} \int_0^{t\land\tau_\alpha}\E\Big(\frac{1}{N}\sum_{i=1}^N|(X_{N,i}^\varepsilon- \overline{X}_i^\varepsilon)(s)|^{2k}\Big)ds\nonumber\\
\le\ & 
C\varepsilon^{-3k} \int_0^{t\land\tau_\alpha}\E\Big(\frac{1}{N}\sum_{i=1}^N N^{-\alpha k}|(X_{N,i}^\varepsilon- \overline{X}_i^\varepsilon)(s)|^{k}\Big)ds\nonumber\\
=\ & 
C\varepsilon^{-3k} N^{-\alpha k}\int_0^{t\land\tau_\alpha}\E(Q_\alpha^{k}(s))ds.
\end{align}
For the second inequality, according to the definition of $\tau_\alpha$ in \eqref{stopingtime}, we note that for any $s \le \tau_\alpha$, $\max_{i=1\dots N} |(X_{N,i}^\varepsilon - \overline{X}_i^\varepsilon)(s)| \le N^{-\alpha}$. Thus,
\begin{align}
|(X_{N,i}^\varepsilon- \overline{X}_i^\varepsilon)(s)|^{2k} 
&= |(X_{N,i}^\varepsilon- \overline{X}_i^\varepsilon)(s)|^k \cdot |(X_{N,i}^\varepsilon- \overline{X}_i^\varepsilon)(s)|^k \nonumber\\
&\le N^{-\alpha k} |(X_{N,i}^\varepsilon- \overline{X}_i^\varepsilon)(s)|^k.
\end{align}

For $I_{11}(t)$, it can be written as
\begin{align}\label{I11+}
I_{11}(t)\le\ & 4\E\Big(\int_0^{t\land\tau_\alpha}\frac{1}{N}\sum_{i=1}^N\Big|\frac{1}{N}\sum_{j=1}^N\nabla\Phi^\varepsilon(\overline{X}_i^\varepsilon - \overline{X}_j^\varepsilon)(X_{N,i}^\varepsilon - \overline{X}_i^\varepsilon )(s) \Big|^kds\Big) \nonumber\\
&+ 4\E\Big(\int_0^{t\land\tau_\alpha}\frac{1}{N}\sum_{i=1}^N\Big|\frac{1}{N}\sum_{j=1}^N\nabla\Phi^\varepsilon(\overline{X}_i^\varepsilon -\overline{X}_j^\varepsilon) (X_{N,j}^\varepsilon - \overline{X}_j^\varepsilon)(s) \Big|^kds\Big)\nonumber\\
=: & C(P_1(t) +P_2(t)).
\end{align}
First, we can verify $P_1(t)$
\begin{align}\label{I111+ }
P_1(t)\le\ &  C \E\Big(\int_0^{t\land\tau_\alpha}\frac{1}{N}\sum_{i=1}^N\Big(\Big|\frac{1}{N}\sum_{j=1}^N\nabla\Phi^\varepsilon(\overline{X}_i^\varepsilon - \overline{X}_j^\varepsilon)(s) \Big|^k \Big|(X_{N,i}^\varepsilon - \overline{X}_i^\varepsilon )(s) \Big|^k\Big)ds\Big)\nonumber\\
\le\ &C\E\Big(\int_0^{t\land\tau_\alpha}Q_\alpha^k(s)\frac{1}{N}\sum_{i=1}^N\Big|\frac{1}{N}\sum_{j=1}^N\nabla\Phi^\varepsilon(\overline{X}_i^\varepsilon - \overline{X}_j^\varepsilon)(s) - (\nabla\Phi^\varepsilon* u^\varepsilon)(\overline{X}_i^\varepsilon(s),s)\Big|^kds\Big) \nonumber\\ 
&+ C\E\Big(\int_0^{t\land\tau_\alpha}Q_\alpha^k(s)\frac{1}{N}\sum_{i=1}^N\Big|(\nabla\Phi^\varepsilon* u^\varepsilon)(\overline{X}_i^\varepsilon(s),s)\Big|^kds\Big) \nonumber\\ 
        =:\ &C(P_{11}(t) + P_{12}(t)).
\end{align}
We proceed to estimate each term individually. For the first term, it can be handled by employing Lemma \ref{law of large number} with $\psi^\varepsilon=\nabla\Phi^\varepsilon$ and $\theta=0$.
\begin{align}\label{I1111}
P_{11}(t) \le\ &\E\Big(\int_0^{t\land\tau_\alpha}Q_\alpha^k(s)\frac{1}{N}\sum_{i=1}^N\Big|\frac{1}{N}\sum_{j=1}^N\nabla\Phi^\varepsilon(\overline{X}_i^\varepsilon - \overline{X}_j^\varepsilon)(s) - (\nabla\Phi^\varepsilon* u^\varepsilon)(\overline{X}_i^\varepsilon
(s),s)\Big|^k \mathbf{1}_{(\mathcal{B}_{0,\nabla\Phi^\varepsilon}^N)^c}ds\Big) \nonumber\\ 
&+ \E\Big(\int_0^{t\land\tau_\alpha}Q_\alpha^k(s)\frac{1}{N}\sum_{i=1}^N\Big|\frac{1}{N}\sum_{j=1}^N\nabla\Phi^\varepsilon(\overline{X}_i^\varepsilon - \overline{X}_j^\varepsilon)(s) - (\nabla\Phi^\varepsilon* u^\varepsilon)(\overline{X}_i^\varepsilon(s),s)\Big|^k \mathbf{1}_{\mathcal{B}_{0,\nabla\Phi^\varepsilon}^N}ds\Big) \nonumber\\ 
\le\ & N^{-0}\int_0^{t\land\tau_\alpha} \E(Q_\alpha^k(s))ds + C(T)\|\nabla\Phi^\varepsilon\|_{L^\infty(\R^2)}^k\sup_{0<s<T}\mathbb{P}(\mathcal{B}_{0,\nabla\Phi^\varepsilon}^N(s))\nonumber\\
\le\ &\int_0^{t\land\tau_\alpha} \E(Q_\alpha^k(s))ds + C(m, T)\varepsilon^{-2k-4m}N^{-m+1}.
\end{align}
Then for $P_{12}(t)$, according to
\begin{align}\label{nabphi_infty}
\|\nabla\Phi^\varepsilon*u^\varepsilon\|_{L^\infty((0,T)\times\R^2)}\le \|\nabla\Phi^\varepsilon\|_{L^1(\R^2)}\|u^\varepsilon\|_{L^\infty((0,T)\times\R^2)}\le C(T),
\end{align}
we have
\begin{align}\label{I1112}
P_{12}(t)=\ & \E\Big(\int_0^{t\land\tau_\alpha}Q_\alpha^k(s)\frac{1}{N}\sum_{i=1}^N\Big|(\nabla\Phi^\varepsilon* u^\varepsilon)(\overline{X}_i^\varepsilon(s),s)\Big|^kds\Big) \nonumber\\ 
\le\ & \|\nabla\Phi^\varepsilon*u^\varepsilon\|_{L^\infty((0,T)\times\R^2)}^k\int_0^{t\land\tau_\alpha}\E(Q_\alpha^k(s))ds\nonumber\\
\le\ & C(T) \int_0^{t\land\tau_\alpha}\E(Q_\alpha^k(s))ds.
\end{align}
Substituting \eqref{I1111} and \eqref{I1112} into \eqref{I111+ }, we obtain
\begin{align}\label{I111}
P_1(t)\le C(T)\int_0^{t\land\tau_\alpha}\E(Q_\alpha^k(s))ds + C(m,T)\varepsilon^{-2k-4m}N^{-m+1}.
\end{align}
Finally, we estimate $P_2(t)$ by first taking the modulus of $X_{N,j}^\varepsilon(s) - \overline{X}_j^\varepsilon(s)$ inside the sum:
\begin{align}
P_2(t) 
\le\ &\E\Big(\int_0^{t\land\tau_\alpha}\frac{1}{N}\sum_{i=1}^N\Big|\frac{1}{N}\sum_{j=1}^N|\nabla\Phi^\varepsilon(\overline{X}_i^\varepsilon -\overline{X}_j^\varepsilon)(s)||(X_{N,j}^\varepsilon - \overline{X}_j^\varepsilon)(s)| \Big|^kds\Big)\nonumber\\
\le\ & 
\E\Big(\int_0^{t\land\tau_\alpha}Q_\alpha^k(s)\frac{1}{N}\sum_{i=1}^N\Big|\frac{1}{N}\sum_{j=1}^N|\nabla\Phi^\varepsilon(\overline{X}_i^\varepsilon - \overline{X}_j^\varepsilon)(s)|\Big|^kds\Big)\nonumber\\
\le\ &C\E\Big(\int_0^{t\land\tau_\alpha}Q_\alpha^k(s)\frac{1}{N}\sum_{i=1}^N\Big|\frac{1}{N}\sum_{j=1}^N|\nabla\Phi^\varepsilon(\overline{X}_i^\varepsilon - \overline{X}_j^\varepsilon)(s)|- (|\nabla\Phi^\varepsilon|* u^\varepsilon)(\overline{X}_i^\varepsilon(s),s)\Big|^kds\Big) \nonumber\\ 
&+ C\E\Big(\int_0^{t\land\tau_\alpha}Q_\alpha^k(s)\frac{1}{N}\sum_{i=1}^N\Big|(|\nabla\Phi^\varepsilon|* u^\varepsilon)(\overline{X}_i^\varepsilon(s),s)\Big|^kds\Big) \nonumber\\ 
=:\ &C(P_{21}(t) + P_{22}(t)).
\end{align}
The Lemma \ref{law of large number} with $\psi^\varepsilon=|\nabla\Phi^\varepsilon|$ and $\theta=0$ implies that
\begin{align}\label{I1121}
P_{21}(t)  
\le\ & N^{-0}\int_0^{t\land\tau_\alpha} \E(Q_\alpha^k(s))ds + C(T)\||\nabla\Phi^\varepsilon|\|_{L^\infty((0,T)\times\R^2)}^k\sup_{0<s<T}\mathbb{P}(\mathcal{B}_{0,|\nabla\Phi^\varepsilon|}^N(s))\nonumber\\
\le\ &\int_0^{t\land\tau_\alpha} \E(Q_\alpha^k(s))ds + C(m, T)\varepsilon^{-2k-4m}N^{-m+1}.
\end{align}
For $P_{22}(t)$, by the same argument as $P_{12}(t)$ in \eqref{I1112}, we have
\begin{align}\label{I1122}
P_{22}(t)=\ & \E\Big(\int_0^{t\land\tau_\alpha}Q_\alpha^k(s)\max_{i=1,\cdots,N}\Big|(|\nabla\Phi^\varepsilon|* u^\varepsilon)(\overline{X}_i^\varepsilon(s),s)\Big|^kds\Big) \nonumber\\ 
\le\ & \| |\nabla\Phi^\varepsilon|*u^\varepsilon\|_{L^\infty((0,T)\times\R^2)}^k\int_0^{t\land\tau_\alpha}\E(Q_\alpha^k(s))ds \le C\int_0^{t\land\tau_\alpha}\E(Q_\alpha^k(s))ds.
\end{align}
Substituting \eqref{I111}, \eqref{I1121}, and \eqref{I1122} into \eqref{I11+}, we obtain
\begin{align}\label{I11}
I_{11}(t)\le C(T)\int_0^{t\land\tau_\alpha}\E(Q_\alpha^k(s))ds + C(m,T)\varepsilon^{-2k-4m}N^{-m+1}.
\end{align}

Therefore, plugging \eqref{I12} and \eqref{I11} into \eqref{EI1i} gives
\begin{align}\label{I1i}
&\E\Big(\int_0^{t\land\tau_\alpha}\frac{1}{N}\sum_{i=1}^N|I_{1,i}(s)|^k ds\Big)\nonumber\\
\le\ & C(T) (1+\varepsilon^{-3k}N^{-\alpha k})\int_0^{t\land\tau_\alpha}\E(Q_\alpha^k(s))ds + C(m, T)\varepsilon^{-2k-4m}N^{-m+1}.
\end{align}
Substituting \eqref{I2i} and \eqref{I1i} into \eqref{ES} yields the inequality
\begin{align}\label{initial ineq}
\E(Q_\alpha^k(t))
\le\ &  C(T) (1+\varepsilon^{-3k}N^{-\alpha k})\int_0^{t\land\tau_\alpha}\E(Q_\alpha^k(s))ds \nonumber\\
&+ C(m, T)(\varepsilon^{-2k-4m}N^{-m+1} +N^{-k\theta} + \varepsilon^{-2k-4m}N^{2m(\theta -1/2)+1}).
\end{align}

We now show that the terms multiplying the constants decay as a negative power of $N$. Recall the scaling $\varepsilon\sim N^{-\gamma},\gamma>0$. First, to ensure the coefficient in the Gr\"{o}nwall integral remains bounded by a constant independent of $N$, we require the exponent to be non-positive:
\begin{align*}
\varepsilon^{-3k}N^{-\alpha k} = N^{3\gamma k - \alpha k}\le N^{0}&\Longleftrightarrow\gamma\le\frac{\alpha}{3}.
\end{align*}

Next, we bound the remaining error terms. After applying Gr\"{o}nwall's inequality to \eqref{initial ineq}, we must multiply the result by $N^{1+\alpha k}$ as dictated by \eqref{MK}. We require a positive exponent $\eta>0$ such that the final multiplied error terms satisfy the following bounds:
{\footnotesize\begin{align*}
N^{1+\alpha k} N^{-k\theta} \le N^{-\eta}&\Longleftrightarrow\eta\le k(\theta-\alpha) - 1, \\
N^{1+\alpha k} \varepsilon^{-2k-4m}N^{-m+1} = N^{\gamma(2k+4m) -m + \alpha k +2}\le N^{-\eta}
&\Longleftrightarrow\eta\le-\gamma(2k+4m) - \alpha k + m - 2,\\
N^{1+\alpha k} \varepsilon^{-2k-4m}N^{2m(\theta -1/2)+1} = N^{\gamma(2k+4m) + m(2\theta -1) + \alpha k + 2}\le N^{-\eta}&\Longleftrightarrow\eta\le-\gamma(2k+4m) -\alpha k +m(1-2\theta)-2.
\end{align*}}
Notice that since $0 < \theta < 1/2$, the third condition is strictly tighter than the second. The condition $\eta>0$ implies that $k$, $m$, and $\gamma$ need to be selected to meet the following relationships. First, to ensure $k(\theta-\alpha) - 1 > 0$, we must pick an integer $k$ satisfying:
\begin{align*}
k > \frac{1}{\theta-\alpha}.
\end{align*}
Having fixed such a $k$, the parameters $m$ and $\gamma$ must be chosen such that:
\begin{align*}
0<\theta<\frac{1}{2},\quad 0<\alpha<\frac{\theta}{2},\quad m>\frac{\alpha k + 2}{1-2\theta},\quad 0<\gamma<\frac{-\alpha k +m(1-2\theta)-2}{2k+4m}. 
\end{align*}
Thus choose $\gamma$ satisfying
$$
0<\gamma<\min\Big\{\frac{\alpha}{3},\frac{-\alpha k +m(1-2\theta)-2}{2k+4m}\Big\},
$$
and $\eta$ satisfying 
$$
0<\eta\le\min\big\{k(\theta-\alpha) - 1, -\gamma(2k+4m) -\alpha k +m(1-2\theta)-2\big\}.
$$ 
Under these conditions, the inequality \eqref{initial ineq} simplifies to the estimate
\begin{align}\label{Sak}
\E(Q_\alpha^k(t))
\le\ &  C(T)\int_0^{t\land\tau_\alpha}\E(Q_\alpha^k(s))ds + C(m,k,T)N^{-\eta-1-\alpha k}.
\end{align}
Using Gr\"{o}nwall's inequality for \eqref{Sak}, we obtain for any $t\in[0,T]$
\begin{align} \label{Egronwall}
\E(Q_\alpha^k(t))\le C(m,k,T)N^{-\eta-1-\alpha k} e^{C(T)t}.  
\end{align}
By inserting \eqref{MK} into \eqref{Egronwall}, we multiply by $N^{1+\alpha k}$ and derive 
\begin{align}
\mathbb{P}\big(\max_{i=1,\cdots,N}\big|(X_{N,i}^\varepsilon - \overline{X}_i^\varepsilon)(t)\big|> N^{-\alpha}\big) \le N^{1+\alpha k}\E(Q_\alpha^k(t)) \le C(m,k,T)e^{C(T)T} N^{-\eta}.
\end{align}
The proof of this theorem is completed.

\end{proof}

\section{$L^1$ convergence for the propagation of chaos}
In this section, we prove Theorem \ref{Propagation of chaos in the strong sense}, which establishes the propagation of chaos in the strong sense. We begin by introducing two key associated PDEs.
Applying It\^o's formula reveals that the associated probability densities of the particle system are governed by the PDEs. Specifically, the interacting particle system \eqref{sde} induces the following Liouville equation (the Kolmogorov forward equation) with solution $u_N^\varepsilon(t, x_1,\cdots,x_N)$ on $\R^{2N}\times[0,T), \forall~ T>0$:
\begin{align}\label{lpde1}
\begin{cases}
\partial_t u_N^\varepsilon
=\sum_{i=1}^N\Delta_{x_i}\Big(\exp\Big(-\frac{1}{N}\sum_{j=1}^N\Phi^\varepsilon(x_i - x_j)\Big)u_N^\varepsilon + u_N^\varepsilon\Big),\\
u_N^\varepsilon(0,x_1,\cdots,x_N) = u_0^{\otimes N}=u_0(x_1)\cdots u_0(x_N).
\end{cases}
\end{align}
The global existence and uniqueness of classical solutions to the linear parabolic problem \eqref{lpde1}, for fixed $0<\varepsilon<1$, follow directly from classical parabolic theory \cite[Chapter 7]{evans2022partial}. 
Additionally, we define the chaotic law ($N$-fold tensor product of $u^\varepsilon$)
$$u^{\varepsilon\otimes N}:=(u^{\varepsilon})^{\otimes N}(t, x_1,\cdots,x_N):= \prod_{i=1}^{N}u^{\varepsilon}(t, x_i),$$
where $u^\varepsilon$ is the solution to \eqref{rpde}. With this definition, one can verify that $u^{\varepsilon\otimes N}$ satisfies the following PDE
\begin{align}\label{lpde2}
\begin{cases}
\partial_t u^{\varepsilon\otimes N}
=\sum_{i=1}^N\Delta_{x_i}(\exp(-\Phi^\varepsilon*u^\varepsilon(t,x_i))u^{\varepsilon\otimes N} + u^{\varepsilon\otimes N}),\\
u^{\varepsilon\otimes N}(0,x_1,\cdots,x_N)=u_0(x_1)\cdots u_0(x_N).
\end{cases}
\end{align}

Then, we introduce the relative entropy, following the definition provided in \cite[Lemma 3.9]{miclo2001genealogies}. For any $l\in \mathbb{N}$ and two probability density functions $p,q:\mathbb{R}^{2l}\to\mathbb{R}^+$, the relative entropy is defined as 
\begin{align}\label{Hm}
\mathcal{H}_l(p|q):=\int_{\R^{2l}}p\log\frac{p}{q}dx_1\cdots dx_l.
\end{align}
We introduce the relative entropy between $u_N^\varepsilon$
and $u^{\varepsilon\otimes N}$, defined as
\begin{align}\label{relative entropy}
\mathcal{H}(u_N^\varepsilon|u^{\varepsilon\otimes N}):=\frac{1}{N}\mathcal{H}_N(u_N^\varepsilon|u^{\varepsilon\otimes N})=\frac{1}{N}\int_{\R^{2N}}u_N^\varepsilon\log\frac{u_N^\varepsilon}{u^{\varepsilon\otimes N}}dx_1\cdots dx_N.
\end{align}
 
Moreover, for any $0<r<N$, the $r$-th marginal is defined as
$$u_{N,r}^{\varepsilon}(t,x_1,\cdots,x_N) :=\int_{\R^{N-r}}u_N^{\varepsilon}(t,x_1,\cdots,x_N)dx_{r+1}\cdots dx_{N}.$$
A key tool for the error estimate between the $r$-th marginal of the solutions to PDE \eqref{lpde1} and the $r$-fold tensor product of the solutions to \eqref{rpde} is the Csiszár-Kullback-Pinsker (C-K-P) inequality \cite{villani2002review}. This inequality demonstrates that their  $L^1(\mathbb{R}^{2r})$ norm of the error estimate can be controlled by relative entropy, i.e., for any $T>0$,
\begin{align}\label{CKP}
\sup_{0<t<T}\|u_{N,r}^{\varepsilon}(t) - u^{\varepsilon\otimes r}(t)\|_{L^1(\R^{2r})}^2 
&\le \sup_{0<t<T}2\mathcal{H}_r(u_{N,r}^{\varepsilon}|u^{\varepsilon\otimes r})(t)\nonumber\\
&\le  \sup_{0<t<T}4r\mathcal{H}(u_{N}^{\varepsilon}|u^{\varepsilon\otimes N})(t).
\end{align}

We now present a rigorous proof of Theorem \ref{Propagation of chaos in the strong sense} 
based on the relative entropy method and the convergence of particle trajectories in probability.
Given that we aim to establish the quantitative propagation of chaos in the norm $L^\infty(0,T^*;L^1(\R^{2r}))$, it is essential to obtain a uniform $L^\infty(0,T^*;H^1(\R^{2}))$ bound for $\nabla\log u^\varepsilon$ in Lemma \ref{weak solution}. Although the statement of Theorem \ref{Propagation of chaos in the strong sense} was provided in the introduction, we restate it in greater detail below for the reader's convenience and provide the specifics of the proof along with relevant derivations. 

\setcounter{theorem}{1}
\begin{theorem}\label{Propagation of chaos in the strong sense}\textup{(Propagation of chaos in the strong sense).} 
Under the assumptions of Lemma \ref{weak solution}, let 
$u_{N}^\varepsilon(t,x_1,\cdots,x_N)$ and $u^{\varepsilon\otimes N}(t,x_1,\cdots,x_N)$ be the solutions of systems \eqref{lpde1} and \eqref{lpde2}, respectively. For any 
$r\in\mathbb{N^+}$, denote by $u_{N,r}^\varepsilon(t,x_1,\cdots,x_r)$ and $u^{\varepsilon\otimes r}(t,x_1,\cdots,x_r)$ their corresponding 
$r$-th marginal distributions. 
Then for parameters $0<\theta<1/2$, $0 <\alpha<\frac{\theta}{2}$, and integers $k,m\in \mathbb{N}^+$ satisfying $k > \frac{1}{\theta-\alpha}$ and $m > \frac{\alpha k+2}{1-2\theta}$, there exist a time $T^*\in(0,T)$, a constant $C(r,m,k,T)>0$, and a parameter $\beta>0$ satisfying the bound $1<\beta\le\min\{\frac{2\alpha}{\gamma}-6,\frac{\eta}{\gamma}-4\}$ such that 
\begin{align}\label{Ps}
\|u_{N,r}^{\varepsilon}(t) - u^{\varepsilon\otimes r}(t)\|_{L^\infty(0,T^*;L^1(\R^{2r}))}^2 \le C(r,m,k,T)\varepsilon^{\beta}, 
  \end{align}

where the cut-off parameter satisfies
$\varepsilon\sim N^{-\gamma} $,
and the parameter $\gamma$ is defined by
$$0<\gamma<\min\Big\{\frac{2\alpha}{7},\frac{-\alpha k + m(1-2\theta)-2}{2k+4m}\Big\},$$
and the parameter $\eta$ satisfies
$$5\gamma<\eta<\min\big\{k(\theta-\alpha)-1, -\gamma(2k+4m)-\alpha k +m(1-2\theta)-2\big\}.$$ 
\end{theorem}


\begin{proof}
According to the definition of relative entropy \eqref{Hm}, by differentiating it with respect to time, substituting \eqref{lpde1} and \eqref{lpde2}, and applying integration by parts, we obtain the following result
\begin{align}\label{H1}
\frac{d}{dt}\mathcal{H}(u_N^\varepsilon|u^{\varepsilon\otimes N})
=& \frac{1}{N}\int_{\R^{2N}}\Big(\partial_t u_N^\varepsilon(\log\frac{u_N^\varepsilon}{u^{\varepsilon\otimes N}}+1) - \frac{u_N^\varepsilon}{u^{\varepsilon\otimes N}}\partial_t u^{\varepsilon\otimes N}\Big)d_{x_1}\cdots d_{x_N}\nonumber\\
=& -\int_{\R^{2N}}\frac{1}{N}\sum_{i=1}^N\Big(\nabla_{x_i}u_N^\varepsilon\cdot\nabla_{x_i}\log\frac{u_N^\varepsilon}{u^{\varepsilon\otimes N}} - \nabla_{x_i}u^{\varepsilon\otimes N}\cdot\nabla_{x_i}\frac{u_N^\varepsilon}{u^{\varepsilon\otimes N}}\Big)d_{x_1}\cdots d_{x_N}\nonumber\\
&-\int_{\R^{2N}}\Big[\frac{1}{N}\sum_{i=1}^N\Big(\nabla_{x_i}\Big(\exp\big(-\frac{1}{N}\sum_{j=1}^N\Phi^\varepsilon(x_i - x_j)\big)u_N^\varepsilon \Big)\cdot\nabla_{x_i}\log\frac{u_N^\varepsilon}{u^{\varepsilon\otimes N}}\Big)\nonumber \\
&\qquad- \frac{1}{N}\sum_{i=1}^N\Big(\nabla_{x_i}\big(\exp(-\Phi^\varepsilon*u^\varepsilon(x_i,t))u^{\varepsilon\otimes N}\big)\cdot\nabla_{x_i}\frac{u_N^\varepsilon}{u^{\varepsilon\otimes N}}\Big)\Big]d_{x_1}\cdots d_{x_N}\nonumber\\
=:& K_1 + K_2.
\end{align}
Since it follows from a direct computation that
\begin{align*}
&\nabla_{x_i}\frac{u_N^\varepsilon}{u^{\varepsilon\otimes N}} = \frac{u_N^\varepsilon}{u^{\varepsilon\otimes N}}\nabla_{x_i}\log\frac{u_N^\varepsilon}{u^{\varepsilon\otimes N}},~~\text{and}~\nabla_{x_i}u_N^\varepsilon - \nabla_{x_i}u^{\varepsilon\otimes N}\frac{u_N^\varepsilon}{u^{\varepsilon\otimes N}} = u_N^\varepsilon\nabla_{x_i}\log\frac{u_N^\varepsilon}{u^{\varepsilon\otimes N}}, 
\end{align*}
then $K_1$ and $K_2$ can be written as
\begin{align}\label{K1}
K_1=& -\int_{\R^{2N}}\frac{1}{N}\sum_{i=1}^N\Big(\big(\nabla_{x_i}u_N^\varepsilon - \nabla_{x_i}u^{\varepsilon\otimes N}\frac{u_N^\varepsilon}{u^{\varepsilon\otimes N}}\big)\cdot\nabla_{x_i}\log\frac{u_N^\varepsilon}{u^{\varepsilon\otimes N}}\Big)d_{x_1}\cdots d_{x_N}\nonumber\\
=& -\int_{\R^{2N}}\frac{1}{N}\sum_{i=1}^Nu_N^\varepsilon\Big|\nabla_{x_i}\log\frac{u_N^\varepsilon}{u^{\varepsilon\otimes N}}\Big|^2d_{x_1}\cdots d_{x_N},
\end{align}
and
\begin{align*}
K_2
=
&-\int_{\R^{2N}}\frac{1}{N}\sum_{i=1}^N\Big[\Big(\nabla_{x_i}\big(\exp\big(-\frac{1}{N}\sum_{j=1}^N\Phi^\varepsilon(x_i - x_j)\big)u_N^\varepsilon\big)\nonumber\\
&\qquad\qquad\qquad -\frac{u_N^\varepsilon}{u^{\varepsilon\otimes N}}\nabla_{x_i}\big(\exp(-\Phi^\varepsilon*u^\varepsilon(x_i,t))u^{\varepsilon\otimes N}\big)\Big)\cdot\nabla_{x_i}\log\frac{u_N^\varepsilon}{u^{\varepsilon\otimes N}}\Big]d_{x_1}\cdots d_{x_N}.
\end{align*}
For $K_2$, the simple calculation gives
\begin{align}\label{K2}
K_2
=
&-\int_{\R^{2N}}\frac{1}{N}\sum_{i=1}^Nu_N^\varepsilon\Big[\exp\big(-\frac{1}{N}\sum_{j=1}^N\Phi^\varepsilon(x_i - x_j)\big)\nonumber\\
&\qquad\qquad\qquad\cdot\Big(-\frac{1}{N}\sum_{j=1}^N\nabla_{x_i}\Phi^\varepsilon(x_i - x_j) + \nabla_{x_i}\log u_N^\varepsilon\Big) \cdot\nabla_{x_i}\log\frac{u_N^\varepsilon}{u^{\varepsilon\otimes N}}\Big]d_{x_1}
\cdots d_{x_N}\nonumber\\
&+\int_{\R^{2N}}\frac{1}{N}\sum_{i=1}^Nu_N^\varepsilon\Big[\exp(-\Phi^\varepsilon*u^\varepsilon(x_i,t))\nonumber\\
&\qquad\qquad\qquad \cdot\big(-\nabla_{x_i}\Phi^\varepsilon*u^\varepsilon(x_i,t)+\nabla_{x_i}\log u^{\varepsilon\otimes N}\big)\cdot\nabla_{x_i}\log\frac{u_N^\varepsilon}{u^{\varepsilon\otimes N}}\Big]d_{x_1}\cdots d_{x_N}\nonumber\\
=:&K_{21} + K_{22}.
\end{align}
By adding and subtracting a term, the first expression $K_{21}$ leads to the following result.
\begin{align}\label{RHS1}
K_{21} = &-\int_{\R^{2N}}\frac{1}{N}\sum_{i=1}^Nu_N^\varepsilon\Big[\exp\big(-\frac{1}{N}\sum_{j=1}^N\Phi^\varepsilon(x_i - x_j)\big) \Big(\nabla_{x_i}\log u_N^\varepsilon - \nabla_{x_i}\log u^{\varepsilon\otimes N}\nonumber\\
&\qquad\qquad\qquad\qquad\qquad+ \nabla_{x_i}\log u^{\varepsilon\otimes N} -\frac{1}{N}\sum_{j=1}^N\nabla_{x_i}\Phi^\varepsilon(x_i - x_j) \Big) \cdot\nabla_{x_i}\log\frac{u_N^\varepsilon}{u^{\varepsilon\otimes N}}\Big]d_{x_1}
\cdots d_{x_N}\nonumber\\
= &-\int_{\R^{2N}}\frac{1}{N}\sum_{i=1}^Nu_N^\varepsilon\Big[\exp\big(-\frac{1}{N}\sum_{j=1}^N\Phi^\varepsilon(x_i - x_j)\big) \Big|\nabla_{x_i}\log\frac{u_N^\varepsilon}{u^{\varepsilon\otimes N}}\Big|^2\Big]d_{x_1}
\cdots d_{x_N}\nonumber\\
&-\int_{\R^{2N}}\frac{1}{N}\sum_{i=1}^Nu_N^\varepsilon\Big[\exp\big(-\frac{1}{N}\sum_{j=1}^N\Phi^\varepsilon(x_i - x_j)\big) \Big(\nabla_{x_i}\Phi^\varepsilon*u^\varepsilon(x_i,t) -\frac{1}{N}\sum_{j=1}^N\nabla_{x_i}\Phi^\varepsilon(x_i - x_j)\nonumber\\
&\qquad\qquad\qquad
+ \nabla_{x_i}\log u^{\varepsilon\otimes N} - \nabla_{x_i}\Phi^\varepsilon*u^\varepsilon(x_i,t)\Big) \cdot\nabla_{x_i}\log\frac{u_N^\varepsilon}{u^{\varepsilon\otimes N}}\Big]d_{x_1}
\cdots d_{x_N}.
\end{align}

Substituting \eqref{RHS1} into \eqref{K2}, we have
\begin{align}
K_2=&-\int_{\R^{2N}}\frac{1}{N}\sum_{i=1}^Nu_N^\varepsilon\Big[\exp\big(-\frac{1}{N}\sum_{j=1}^N\Phi^\varepsilon(x_i - x_j)\big) \Big|\nabla_{x_i}\log\frac{u_N^\varepsilon}{u^{\varepsilon\otimes N}}\Big|^2\Big]d_{x_1}
\cdots d_{x_N}\nonumber\\
&+J_1+J_2,\label{K21}
\end{align}
where $J_1$ and $J_2$ are defined respectively by
\begin{align}
J_1:= &-\int_{\R^{2N}}\frac{1}{N}\sum_{i=1}^Nu_N^\varepsilon\Big[\exp\big(-\frac{1}{N}\sum_{j=1}^N\Phi^\varepsilon(x_i - x_j)\big) \nonumber\\
&\qquad\qquad\cdot \Big(\nabla_{x_i}\Phi^\varepsilon*u^\varepsilon(x_i,t) -\frac{1}{N}\sum_{j=1}^N\nabla_{x_i}\Phi^\varepsilon(x_i - x_j)\Big) \cdot\nabla_{x_i}\log\frac{u_N^\varepsilon}{u^{\varepsilon\otimes N}}\Big]d_{x_1}
\cdots d_{x_N} \label{J1}
\end{align}
and
\begin{align}
J_2:= &-\int_{\R^{2N}}\frac{1}{N}\sum_{i=1}^Nu_N^\varepsilon\Big[\Big(\exp\big(-\frac{1}{N}\sum_{j=1}^N\Phi^\varepsilon(x_i - x_j)\big) -\exp(-\Phi^\varepsilon*u^\varepsilon(x_i,t))\Big)\nonumber\\
&\qquad\qquad\qquad\qquad\qquad
\cdot\big(\nabla_{x_i}\log u^{\varepsilon\otimes N} - \nabla_{x_i}\Phi^\varepsilon*u^\varepsilon(x_i,t)\big) \cdot\nabla_{x_i}\log\frac{u_N^\varepsilon}{u^{\varepsilon\otimes N}}\Big]d_{x_1}
\cdots d_{x_N}.\label{J2}
\end{align}
For $J_1$, an application of H\"{o}lder's inequality and Young's inequality  gives
\begin{align}\label{J1p}
J_1
\le\ & \frac{1}{2}\int_{\R^{2N}}\frac{1}{N}\sum_{i=1}^Nu_N^\varepsilon\Big[\exp\big(-\frac{1}{N}\sum_{j=1}^N\Phi^\varepsilon(x_i - x_j)\big)\Big|\nabla_{x_i}\log\frac{u_N^\varepsilon}{u^{\varepsilon\otimes N}}\Big|^2\Big]d_{x_1}
\cdots d_{x_N}\nonumber\\ 
&+\frac{1}{2} \Big\|\exp\big(-\frac{1}{N}\sum_{j=1}^N\Phi^\varepsilon(x_i - x_j)\big)\Big\|_{L^\infty(\R^{2})}\nonumber\\
&\qquad\cdot\int_{\R^{2N}}\frac{1}{N}\sum_{i=1}^Nu_N^\varepsilon\Big[\Big|\nabla_{x_i}\Phi^\varepsilon*u^\varepsilon(x_i) -\frac{1}{N}\sum_{j=1}^N\nabla_{x_i}\Phi^\varepsilon(x_i - x_j)\Big|^2\Big]d_{x_1}
\cdots d_{x_N}\nonumber\\
\le\ & \frac{1}{2}\int_{\R^{2N}}\frac{1}{N}\sum_{i=1}^Nu_N^\varepsilon\Big[\exp\big(-\frac{1}{N}\sum_{j=1}^N\Phi^\varepsilon(x_i - x_j)\big)\Big|\nabla_{x_i}\log\frac{u_N^\varepsilon}{u^{\varepsilon\otimes N}}\Big|^2\Big]d_{x_1}
\cdots d_{x_N}\nonumber\\ 
&+C\int_{\R^{2N}}\frac{1}{N}\sum_{i=1}^Nu_N^\varepsilon\Big[\Big|\nabla_{x_i}\Phi^\varepsilon*u^\varepsilon(x_i) -\frac{1}{N}\sum_{j=1}^N\nabla_{x_i}\Phi^\varepsilon(x_i - x_j)\Big|^2\Big]d_{x_1}
\cdots d_{x_N}.
\end{align}
Here, we focus primarily on analyzing the second term on the right-hand side of the inequality
\begin{align}\label{J}
& \int_{\R^{2N}}\frac{1}{N}\sum_{i=1}^Nu_N^\varepsilon\Big[\Big|\nabla_{x_i}\Phi^\varepsilon*u^\varepsilon(x_i,t) -\frac{1}{N}\sum_{j=1}^N\nabla_{x_i}\Phi^\varepsilon(x_i - x_j)\Big|^2\Big]d_{x_1}
\cdots d_{x_N}\nonumber\\  
=\ & \E \Big(\frac{1}{N}\sum_{i=1}^N\Big|\nabla_{x_i}\Phi^\varepsilon*u^\varepsilon(X_{N,i}^\varepsilon(t)) -\frac{1}{N}\sum_{j=1}^N\nabla_{x_i}\Phi^\varepsilon(X_{N,i}^\varepsilon(t) - X_{N,j}^\varepsilon(t))\Big|^2\Big)\nonumber\\
\le\ & 3\E \Big(\frac{1}{N}\sum_{i=1}^N\Big|\nabla_{x_i}\Phi^\varepsilon*u^\varepsilon(X_{N,i}^\varepsilon(t)) -\nabla_{x_i}\Phi^\varepsilon*u^\varepsilon(\overline{X}_i^\varepsilon(t))\Big|^2\Big)\nonumber\\
&+ 3\E \Big(\frac{1}{N}\sum_{i=1}^N\Big|\nabla_{x_i}\Phi^\varepsilon*u^\varepsilon(\overline{X}_i^\varepsilon(t)) -\frac{1}{N}\sum_{j=1}^N\nabla_{x_i}\Phi^\varepsilon(\overline{X}_i^\varepsilon(t) - \overline{X}_j^\varepsilon(t))\Big|^2\Big)\nonumber\\
&+3 \E \Big(\frac{1}{N}\sum_{i=1}^N\Big|\frac{1}{N}\sum_{j=1}^N\nabla_{x_i}\Phi^\varepsilon(\overline{X}_i^\varepsilon(t) - \overline{X}_j^\varepsilon(t)) -\frac{1}{N}\sum_{j=1}^N\nabla_{x_i}\Phi^\varepsilon(X_{N,i}^\varepsilon(t) - X_{N,j}^\varepsilon(t))\Big|^2\Big)\nonumber\\
=:&\, 3(J_{11} + J_{12} + J_{13}).
\end{align}

The second term $J_{12}$ can be directly treated through the Law of Large Numbers in Lemma \ref{law of large number} for $m=1$, leading to
\begin{align}\label{J12}
J_{12} 
\le&C\|\nabla\Phi^\varepsilon\|_{L^\infty(\R^2)}^2N^{-1}\le C\varepsilon^{-4}N^{-1}.
\end{align} 

The estimates of terms $J_{11}$ and $J_{13}$ can be established through probabilistic convergence arguments. For conciseness, we demonstrate the detailed derivation for $J_{13}$, as the analysis for $J_{11}$ follows analogously. Following the notation in Theorem \ref{propagation of chaos}, we define 
$$
\mathcal{A}_\alpha = \big\{\omega\in\Omega:\max_{i=1,\cdots,N}|X_{N,i}^\varepsilon(t) -\overline{X}_i^\varepsilon(t)|>N^{-\alpha}\big\}. $$
That implies that
$$|X_{N,i}^\varepsilon(t) -\overline{X}_i^\varepsilon(t)|\le N^{-\alpha}~~ \text{in}~ (\mathcal{A}_\alpha)^c, $$
which leads to
\begin{align}\label{J13p}
J_{13}=\ &\E \Big(\frac{1}{N}\sum_{i=1}^N\Big|\frac{1}{N}\sum_{j=1}^N\nabla_{x_i}\Phi^\varepsilon(\overline{X}_i^\varepsilon - \overline{X}_j^\varepsilon) -\frac{1}{N}\sum_{j=1}^N\nabla_{x_i}\Phi^\varepsilon(X_{N,i}^\varepsilon - X_{N,j}^\varepsilon)\Big|^2\mathbf{1}_{(\mathcal{A_\alpha})^c}\Big)\nonumber\\
&+ \E \Big(\frac{1}{N}\sum_{i=1}^N\Big|\frac{1}{N}\sum_{j=1}^N\nabla_{x_i}\Phi^\varepsilon(\overline{X}_i^\varepsilon - \overline{X}_j^\varepsilon) -\frac{1}{N}\sum_{j=1}^N\nabla_{x_i}\Phi^\varepsilon(X_{N,i}^\varepsilon - X_{N,j}^\varepsilon)\Big|^2\mathbf{1}_{\mathcal{A_\alpha}}\Big)\nonumber\\
=:& J_{13}^1 +J_{13}^2.
\end{align}
By using H\"{o}lder's inequality and properties of $\Phi^\varepsilon$ in \eqref{phi} and \eqref{D2phi}, we have
\begin{align}\label{J131}
J_{13}^1 \le\ & 2\|D^2\Phi^\varepsilon\|_{L^\infty(\R^2)}^2 \E\big(\frac{1}{N}\sum_{i=1}^N\big|X_{N,i}^\varepsilon(t) -\overline{X}_i^\varepsilon(t)\big|^2\mathbf{1}_{(\mathcal{A_\alpha})^c}\big)\nonumber\\ \le\ & C(T)N^{-2\alpha}\varepsilon^{-6},
\end{align}
and
\begin{align}\label{J132}
J_{13}^2 \le 2\|\nabla\Phi^\varepsilon\|_{L^\infty(\R^2)}^2\sup_{0<s<t}\mathbb{P}(\mathcal{A}_\alpha)
\le C(m, T)N^{-\eta}\varepsilon^{-4}.
\end{align}
By inserting \eqref{J131} and \eqref{J132} into \eqref{J13p}, we derive
\begin{align}\label{J13}
J_{13} \le C(m, T)(N^{-2\alpha}\varepsilon^{-6} + N^{-\eta}\varepsilon^{-4}).
\end{align}

Similarly, $J_{11}$ can be estimated as follows
\begin{align}\label{J11}
J_{11}\le\ & \E \Big(\frac{1}{N}\sum_{i=1}^N\Big|\nabla_{x_i}\Phi^\varepsilon*u^\varepsilon(X_{N,i}^\varepsilon(t)) -\nabla_{x_i}\Phi^\varepsilon*u^\varepsilon(\overline{X}_i^\varepsilon(t))\Big|^2\mathbf{1}_{(\mathcal{A_\alpha})^c}\Big)\nonumber\\
&+ \E \Big(\frac{1}{N}\sum_{i=1}^N\Big|\nabla_{x_i}\Phi^\varepsilon*u^\varepsilon(X_{N,i}^\varepsilon(t)) -\nabla_{x_i}\Phi^\varepsilon*u^\varepsilon(\overline{X}_i^\varepsilon(t))\Big|^2\mathbf{1}_{\mathcal{A_\alpha}}\Big)\nonumber\\
\le\ & 2\|D^2\Phi^\varepsilon\|_{L^\infty(\R^2)}^2 \E\big(\frac{1}{N}\sum_{i=1}^N\big|X_{N,i}^\varepsilon(t) -\overline{X}_i^\varepsilon(t)\big|^2\mathbf{1}_{(\mathcal{A_\alpha})^c}\big)\nonumber\\
&+ 2\|\nabla\Phi^\varepsilon*u^\varepsilon\|_{L^\infty(\R^2)}^2\sup_{0<s<t}\mathbb{P}(\mathcal{A}_\alpha) \nonumber\\
\le\ & C(m,T)(N^{-2\alpha}\varepsilon^{-6} + N^{-\eta}).
\end{align}
Through the substitution of \eqref{J12}, \eqref{J13} and \eqref{J11} in \eqref{J1p} we establish that for $0<\varepsilon<1$,
\begin{align}\label{J1}
J_1
\le\ & \frac{1}{2}\int_{\R^{2N}}\frac{1}{N}\sum_{i=1}^Nu_N^\varepsilon\Big[\exp\big(-\frac{1}{N}\sum_{j=1}^N\Phi^\varepsilon(x_i - x_j)\big)\Big|\nabla_{x_i}\log\frac{u_N^\varepsilon}{u^{\varepsilon\otimes N}}\Big|^2\Big]d_{x_1}
\cdots d_{x_N}\nonumber\\ 
& + C(m, T)( N^{-2\alpha}\varepsilon^{-6} +  (N^{-1} + N^{-\eta})\varepsilon^{-4}).
\end{align}

In what follows, we derive an estimate for $J_2$
\begin{align*}
J_2\le  \big\|\nabla_{x_i}\log u^{\varepsilon\otimes N} - \nabla_{x_i}\Phi^\varepsilon*u^\varepsilon(x_i,t)\big\|_{L^\infty(\R^2)}
&\int_{\R^{2N}}\frac{1}{N}\sum_{i=1}^Nu_N^\varepsilon\Big[\Big|\exp\big(-\frac{1}{N}\sum_{j=1}^N\Phi^\varepsilon(x_i - x_j)\big)\nonumber\\
&\quad-\exp(-\Phi^\varepsilon*u^\varepsilon(x_i))\Big|
\Big|\nabla_{x_i}\log\frac{u_N^\varepsilon}{u^{\varepsilon\otimes N}}\Big|\Big]d_{x_1}
\cdots d_{x_N}.
\end{align*}
Since
\begin{align*}
&\nabla_{x_i}\log u^{\varepsilon\otimes N} = \nabla_{x_i}\log(\prod_{j=1}^Nu^\varepsilon(x_j)) = \nabla_{x_i}(\sum_{j=1}^N\log u^\varepsilon(x_j)) =\nabla_{x_i}\log u^\varepsilon(x_i),\\
&\|\nabla\Phi\|_{L^q(\R^2)}\le
C,~q\in[1,2),\quad\|\nabla\Phi^\varepsilon\|_{L^\frac{4}{3}(\R^2)}\le\|\nabla\Phi\|_{L^\frac{4}{3}(\R^2)}\|j^\varepsilon\|_{L^1(\R^2)}\le C,  
\end{align*}
together with (\ref{uniformweaksolution1}) and the Sobolev embedding theorem, we have for any $0<t<T^*<T$,
\begin{align*}
&\|\nabla_{x_i}\log u^{\varepsilon\otimes N}\|_{L^\infty(\R^2)}\le \|\nabla_{x_i}\log u^\varepsilon\|_{L^\infty(\R^2)}\le C,\\
&\|\nabla\Phi^\varepsilon * u^\varepsilon\|_{L^\infty(\R^2)}\le\|\nabla\Phi^\varepsilon\|_{L^\frac{4}{3}(\R^2)}\|u^\varepsilon\|_{L^4(\R^2)}\le C.
\end{align*}
Following the same approach as in \eqref{J1p} and applying H\"{o}lder's inequality, we obtain
\begin{align}\label{J2p}
J_2
\le\ & \frac{1}{2}\int_{\R^{2N}}\frac{1}{N}\sum_{i=1}^Nu_N^\varepsilon\Big|\nabla_{x_i}\log\frac{u_N^\varepsilon}{u^{\varepsilon\otimes N}}\Big|^2d_{x_1}
\cdots d_{x_N}\nonumber\\ 
&+ C\int_{\R^{2N}}\frac{1}{N}\sum_{i=1}^Nu_N^\varepsilon\Big[\Big|\nabla_{x_i}\Phi^\varepsilon*u^\varepsilon(x_i,t) -\frac{1}{N}\sum_{j=1}^N\nabla_{x_i}\Phi^\varepsilon(x_i - x_j)\Big|^2\Big]d_{x_1}
\cdots d_{x_N}. 
\end{align}

Since the computation here is similar to $J_1$, we directly present the result
\begin{align}\label{J2}
J_2 \le\ & \frac{1}{2}\int_{\R^{2N}}\frac{1}{N}\sum_{i=1}^Nu_N^\varepsilon\Big|\nabla_{x_i}\log\frac{u_N^\varepsilon}{u^{\varepsilon\otimes N}}\Big|^2d_{x_1}
\cdots d_{x_N}\nonumber\\ 
& + C(m, T)(N^{-2\alpha}\varepsilon^{-6} + (N^{-1} + N^{-\eta})\varepsilon^{-4}).
\end{align}

Plugging \eqref{J1} and \eqref{J2} into \eqref{K21} produces the following
\begin{align}
K_2\leq &-\int_{\R^{2N}}\frac{1}{N}\sum_{i=1}^Nu_N^\varepsilon\Big[\exp\big(-\frac{1}{N}\sum_{j=1}^N\Phi^\varepsilon(x_i - x_j)\big) \Big|\nabla_{x_i}\log\frac{u_N^\varepsilon}{u^{\varepsilon\otimes N}}\Big|^2\Big]d_{x_1}
\cdots d_{x_N}\nonumber\\
&+ C(m, T)(N^{-2\alpha}\varepsilon^{-6} + (N^{-1} + N^{-\eta})\varepsilon^{-4}).\label{K_2}
\end{align}
Combining \eqref{K1} and \eqref{H1}, we derive
\begin{align}
\frac{d}{dt}\mathcal{H}(u_N^\varepsilon|u^{\varepsilon\otimes N})
\leq& -\frac{1}{2}\int_{\R^{2N}}\frac{1}{N}\sum_{i=1}^Nu_N^\varepsilon\Big|\nabla_{x_i}\log\frac{u_N^\varepsilon}{u^{\varepsilon\otimes N}}\Big|^2d_{x_1}\cdots d_{x_N}\nonumber\\
&-\frac{1}{2}\int_{\R^{2N}}\frac{1}{N}\sum_{i=1}^Nu_N^\varepsilon\Big[\exp\big(-\frac{1}{N}\sum_{j=1}^N\Phi^\varepsilon(x_i - x_j)\big) \Big|\nabla_{x_i}\log\frac{u_N^\varepsilon}{u^{\varepsilon\otimes N}}\Big|^2\Big]d_{x_1}
\cdots d_{x_N}\nonumber\\
& + C(m, T)(N^{-2\alpha}\varepsilon^{-6} + (N^{-1} + N^{-\eta})\varepsilon^{-4}).\label{entropy}
\end{align}
Noticing that the initial relative entropy is zero, it can be obtained from (\ref{entropy}) that for any $0<t<T^*<T$, 
$$
\mathcal{H}(u_{N}^{\varepsilon}|u^{\varepsilon\otimes N})(t)\le C(m, T)(N^{-2\alpha}\varepsilon^{-6} + (N^{-1} + N^{-\eta})\varepsilon^{-4}).
$$
Combing \eqref{relative entropy} and the C-K-P inequality \eqref{CKP}, we have for $r\in\mathbb{N}$, 
\begin{align}
\sup_{t\in(0,T^*)}\|u_{N,r}^{\varepsilon}(t) - u^{\varepsilon\otimes r}(t)\|_{L^1(\R^{2r})}^2
&\le \sup_{t\in(0,T^*)}2\mathcal{H}_r(u_{N,r}^{\varepsilon}|u^{\varepsilon\otimes r})(t)\nonumber\\
&\le \sup_{t\in(0,T^*)}4r\mathcal{H}(u_{N}^{\varepsilon}|u^{\varepsilon\otimes N})(t)\nonumber\\
&\le C(m, T)(N^{-2\alpha}\varepsilon^{-6} + (N^{-1} + N^{-\eta})\varepsilon^{-4}).
\end{align}

According to Theorem \ref{propagation of chaos}, we know that 
\begin{align}
&0<\eta\le\min\big\{k(\theta-\alpha)-1, -\gamma(2k+4m)-\alpha k +m(1-2\theta)-2\big\},\label{eta1}\\
&0<\gamma<\min\Big\{\frac{\alpha}{3},\frac{-\alpha k +m(1-2\theta)-2}{2k+4m}\Big\}. \label{gamma1} 
\end{align}
Based on the previous assumption $\varepsilon\sim N^{-\gamma}$, it holds that there exists $\beta>0$ such that 
\begin{align}
N^{-2\alpha}\varepsilon^{-6} + (N^{-1} + N^{- \eta})\varepsilon^{-4}=\varepsilon^\frac{2\alpha}{\gamma}\varepsilon^{-6}+(\varepsilon^\frac{1}{\gamma}+\varepsilon^\frac{\eta}{\gamma})\varepsilon^{-4}\le C\varepsilon^\beta, 
\end{align}
where $\beta$ satisfies
$$1<\beta\le\min\Big\{\frac{2\alpha}{\gamma}-6,\frac{\eta}{\gamma}-4\Big\}.$$ 
Here we have assumed $\eta \le 1$ without loss of generality (since if $\eta > 1$, the convergence rate would simply be dominated by the $N^{-1}$ term).
To ensure $\beta>1$, this further requires that $\frac{2\alpha}{\gamma} - 6 > 1 \implies 0<\gamma<\frac{2\alpha}{7}$, and $\frac{\eta}{\gamma} - 4 > 1 \implies \eta>5\gamma$. Building on the assumptions in \eqref{eta1} and \eqref{gamma1}, we assume
\begin{align*}
&0<\gamma<\min\Big\{\frac{2\alpha}{7},\frac{-\alpha k +m(1-2\theta)-2}{2k+4m}\Big\},\\
&5\gamma<\eta<\min\big\{k(\theta-\alpha)-1, -\gamma(2k+4m)-\alpha k +m(1-2\theta)-2\big\}.
\end{align*} 
Then we obtain that for  $0<t<T^*<T$, 
\begin{align}
\sup_{t\in(0,T^*)}\|u_{N,r}^{\varepsilon}(t) - u^{\varepsilon\otimes r}(t)\|_{L^1(\R^{2r})}^2 \le C(m,k,r,T)\varepsilon^\beta,\quad \beta>1.
\end{align}
Thus, we complete the proof of propagation of chaos in the strong sense.
\end{proof}
In addition, we have discovered the following lemma given in \cite{BOL2026113712}.
\begin{lemma}\textup{(\cite[Theorem 2]{BOL2026113712})}
Under the assumptions of Lemma \ref{weak solution}. Then the problems \eqref{pde} and \eqref{rpde} possesses weak solutions $(u,v)$ and $(u^\varepsilon, v^\varepsilon)$ in $[0, T ]$ satisfying
\begin{align}
\|u^\varepsilon-u\|_{L^\infty(0,T;L^1(\R^2))}^2 \le C\varepsilon.
\end{align}
\end{lemma}
Consequently, by the triangle inequality we have the following result.
\begin{corollary}
Under the assumptions of Theorem \ref{Propagation of chaos in the strong sense}, there exists a constant $C > 0$ independent of $\varepsilon$ and $N$ such that
\begin{align}
\|u_{N,1}^{\varepsilon} - u\|_{L^\infty(0,T^*;L^1(\R^{2}))}^2 \le C\varepsilon,
\end{align}
where $\varepsilon \sim N^{-\gamma}$ for some $0<\gamma<\min\Big\{\frac{2\alpha}{7},\frac{-\alpha k + m(1-2\theta)-2}{2k+4m}\Big\}$.
\end{corollary}

\section{Interest statement}
On behalf of all authors, the corresponding author states that there is no conflict of interest.

\bibliographystyle{plain}
\bibliography{PRE2DKS.bib}

\end{document}